\documentclass[a4paper]{article}

\usepackage[utf8]{inputenc}
\usepackage{lmodern}
\usepackage{geometry}
\usepackage{array}
\usepackage{graphicx}
\usepackage{amssymb, amsfonts, amsthm, amsmath}
\usepackage[english]{babel}
\usepackage[pdfborder={0 0 0}]{hyperref}
\usepackage{tikz}
\usetikzlibrary{decorations}
\usetikzlibrary{decorations.pathreplacing}
\tikzset{individu/.style={draw,thick}}
\usepackage{pgfplots}

\usepackage{subfigure}

\theoremstyle{plain}
\newtheorem{theorem}{Theorem}[section]

\newtheorem{lemma}[theorem]{Lemma}
\newtheorem{proposition}[theorem]{Proposition}

\renewcommand{\theotherthm}{\Alph{otherthm}}

\theoremstyle{definition}

\theoremstyle{remark}

\newtheorem{example}[theorem]{Example}

\newcommand{\N}{\mathbb{N}}
\newcommand{\Z}{\mathbb{Z}}
\newcommand{\R}{\mathbb{R}}
\newcommand{\C}{\mathbb{C}}
\newcommand{\calR}{\mathcal{R}}

\newcommand{\ind}[1]{\mathbf{1}_{\left\{#1\right\}}}
\newcommand{\indset}[1]{\mathbf{1}_{#1}}
\newcommand{\floor}[1]{{\left\lfloor #1 \right\rfloor}}
\newcommand{\ceil}[1]{{\left\lceil #1 \right\rceil}}
\newcommand{\norme}[1]{{\left\Vert #1 \right\Vert}}
\newcommand{\calC}{\mathcal{C}}

\numberwithin{equation}{section}

\DeclareMathOperator{\E}{\mathbb{E}}
\newcommand{\Q}{\mathbb{Q}}
\renewcommand{\P}{\mathbb{P}}
\DeclareMathOperator{\Var}{\mathbb{V}\mathrm{ar}}

\newcommand{\bbT}{\mathbb{T}}

\newcommand{\calF}{\mathcal{F}}
\newcommand{\calE}{\mathcal{E}}
\newcommand{\calL}{\mathcal{L}}
\newcommand{\calU}{\mathcal{U}}
\newcommand{\calT}{\mathcal{T}}
\newcommand{\calG}{\mathcal{G}}
\newcommand{\calA}{\mathcal{A}}
\newcommand{\calB}{\mathcal{B}}

\newcommand{\crochet}[1]{{\langle #1 \rangle}}

\newcommand{\e}{\mathbf{e}}
\newcommand{\W}{\mathbb{W}}

\newcommand{\conv}[2]{{\underset{#1\to #2}{\longrightarrow}}}
\newcommand{\wconv}[2]{{\underset{#1\to #2}{\Longrightarrow}}}
\newcommand{\pconv}[2]{{\underset{#1\to #2}{\overset{\mathbb{P}}{\longrightarrow}}}}
\newcommand{\convn}{{\underset{n\to +\infty}{\longrightarrow}}}
\newcommand{\wconvn}{{\underset{n\to +\infty}{\Longrightarrow}}}
\newcommand{\pconvn}{{\underset{n \to +\infty}{\overset{\mathbb{P}}{\longrightarrow}}}}

\renewcommand{\bar}[1]{\overline{#1}}

\newcommand{\egaldistr}{{\overset{(d)}{=}}}
\DeclareMathOperator{\sgn}{sgn}
\newcommand{\esssup}{\mathrm{esssup}}

\title{Asymptotic of the maximal displacement in a branching~random~walk}
\author{Bastien Mallein}
\date{\today}

\newcommand{\T}{\mathbf{T}}
\renewcommand{\tilde}[1]{\widetilde{#1}}
\renewcommand{\hat}[1]{\widehat{#1}}

\renewcommand{\rho}{\varrho}
\renewcommand{\epsilon}{\varepsilon}
\newcommand{\triangleBRW}[2]{	\shade [top color = black!50, bottom color = white] (#1-0.75,-8) -- (#1,-#2-#2-0.1) -- (#1+0.75,-8) -- cycle;
	\draw [dotted] (#1+0.75,-8) -- (#1-0.75,-8);
	\draw [color=blue!50, thick] (#1-0.75,-8) -- (#1,-#2-#2-0.1) -- (#1+0.75,-8);
	\draw [color=blue!50] (#1,-#2-#2) node {$\bullet$} ;
	\draw (#1,-7.5) node {$\P_{\cdot}$} ;
}

\begin{document}

\maketitle

\begin{abstract}
In this article, we study the maximal displacement in a branching random walk. We prove that its asymptotic behaviour consists in a first almost sure ballistic term, a negative logarithmic correction in probability and stochastically bounded fluctuations. This result, proved in \cite{HuS09} and \cite{ABR09} is given here under close-to-optimal integrability conditions. Borrowing ideas from \cite{AiS10} and \cite{Rob16}, we provide simple proofs for this result, also deducing the genealogical structure of the individuals that are close to the maximal displacement.
\end{abstract}

\section{Introduction}
\label{sec:introduction}

A branching random walk on $\R$ is a particle process that evolves as follows. It starts with a unique individual located at the origin at time 0. At each time $k$, each individual alive in the process dies, while giving birth to a random number of children, that are positioned around their parent according to i.i.d. version of a point process. We denote by $\T$ the genealogical tree of the process. For any $u \in \T$, we write $V(u)$ for the position of $u$ and $|u|$ for the generation to which $u$ belongs. The quantity of interest is the maximal displacement $M_n = \max_{|u|=n} V(u)$ at time $n$ in the process.

Under sufficient integrability conditions, the asymptotic behaviour of $M_n$ is fully known. Hammersley \cite{Ham74}, Kingman \cite{Kin75} and Biggins \cite{Big76} proved it grows almost surely at linear speed. In 2009, Hu and Shi \cite{HuS09} exhibited a logarithmic correction in probability, with almost sure fluctuations; and Addario-Berry and Reed \cite{ABR09} showed the tightness of the maximal displacement, shifted around its median. More recently, Aïdékon \cite{Aid13} proved the fluctuations converge in law to some random shift of a Gumbel variable, under integrability conditions that Chen \cite{Che15} proved to be optimal.

Aïdékon and Shi gave in \cite{AiS10} a simple method to obtain the asymptotic behaviour of the maximal displacement. They bounded the maximal displacement by computing the number of individuals that cross a linear boundary. However, the integrability conditions they provided were not optimal, and the fluctuations they obtained were up to $o(\log n)$ order. With a slight refinement of their method, we compute the asymptotic behaviour up to terms of order 1.

We were able to obtain the asymptotic of the maximal displacement up to a term of order 1 by choosing a bended boundary for the study of the branching random walk, a method already used by Roberts \cite{Rob16} for the related model of the branching Brownian motion. This idea follows from a bootstrap heuristic argument, which is explained in Section~\ref{sec:tailbehaviour}. The close-to-optimal integrability conditions arise naturally using the well-known spinal decomposition, introduced by Lyons in \cite{Lyo97}, and recalled in Section~\ref{sec:spinaldecomposition}.

We introduce some notation. In the rest of the article, $c,C$ are two positive constants, respectively small enough and large enough, which may change from line to line, and depend only on the law of the random variables we consider. For a given sequence of random variables $(X_n,n \geq 1)$ we write $X_n = O_\P(1)$ if the sequence is tensed, i.e. $\lim_{K \to +\infty} \sup_{n \geq 1} \P\left( |X_n| \geq K \right) = 0$. Moreover, we always assume the convention $\max \emptyset = -\infty$ and $\min \emptyset = +\infty$. For any $u \in \R$, we write $u_+=\max(u,0)$ and $\log_+ u = (\log u)_+$.

The main assumptions on the branching random walk $(\T,V)$ that we consider are the following. The Galton-Watson tree $\T$ of the branching random walk is supercritical:
\begin{equation}
  \label{eqn:supercritical}
  \E\left( \sum_{|u|=1} 1 \right) > 1 \quad \text{and we write $S = \{\# \T = +\infty\}$ for the survival event}.
\end{equation}
The relative displacements $V(u)$ of the children $u$ are in the so-called (see \cite{BiK05}) boundary case
\begin{equation}
  \label{eqn:boundary}
  \E\left( \sum_{|u|=1} e^{V(u)} \right) = 1 , \, \E\left( \sum_{|u|=1} V(u) e^{V(u)} \right) = 0 \text{ and }  \sigma^2 := \E\left( \sum_{|u|=1} V(u)^2 e^{V(u)} \right) < +\infty.
\end{equation}
For any branching random walk satisfying \eqref{eqn:boundary}, we have $\lim_{n \to +\infty} M_n/n = 0$ a.s. on $S$. Any branching random walk with mild integrability assumption can be normalized to satisfy these inequalities, see e.g. Bérard and Gouéré \cite{BeG11}. We also assume that
\begin{equation}
  \label{eqn:spine}
    \E\left[ \sum_{|u|=1} e^{V(u)} \log_+ \left( \sum_{|v|=1} (1 + V(v)_+) e^{V(v)}\right)^2 \right] < +\infty,
\end{equation}
which is a standard assumption to the study of the maximal displacement.

\begin{figure}[ht]
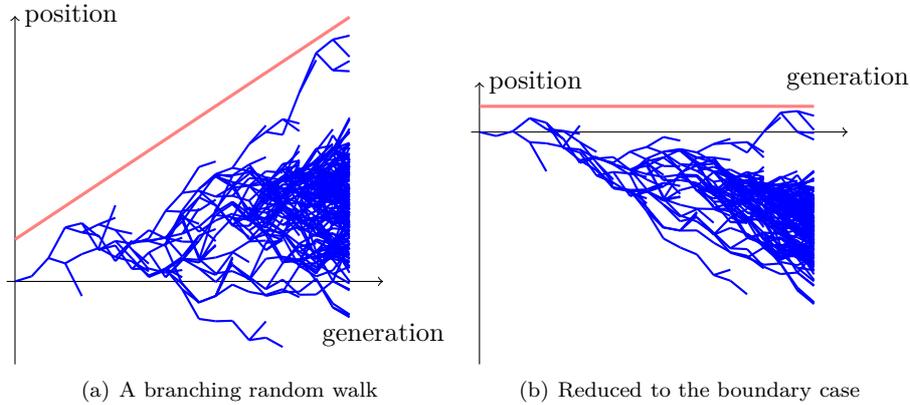

\centering
\subfigure[A branching random walk]{
 
}%
\subfigure[Reduced to the boundary case]{
%
}
\caption{A branching random walk on $\R$}
\end{figure}

The main result of the article is the following.
\begin{theorem}
\label{thm:main}
Under the assumptions \eqref{eqn:supercritical}, \eqref{eqn:boundary} and \eqref{eqn:spine}, we have
\[M_n = - \frac{3}{2} \log n + O_\P(1) \quad \text{ on } S.\]
\end{theorem}
As a side result (see Theorem~\ref{thm:genealogy}), we prove the following well-known fact: the individuals with the largest displacements at time $n$ are either close relatives, or their most recent common ancestor is a close relative of the root (see \cite{ABBS13,ABK11} for similar results for the branching Brownian motion).

The rest of the article is organised as follows. In Section~\ref{sec:spinaldecomposition}, we introduce the so-called spinal decomposition of the branching random walk, which links additive moments of the branching random walk with random walk estimates. In Section~\ref{sec:randomwalk}, we give a list of well-known random walk estimates, and extend them to random walks enriched with additional random variables. Section~\ref{sec:tailbehaviour} is devoted to the study of the left tail of $M_n$. This result is used in Section~\ref{sec:conclusion} to prove Theorem~\ref{thm:main}, using a coupling between the branching random walk and a Galton-Watson process.

\section{Spinal decomposition of the branching random walk}
\label{sec:spinaldecomposition}

We introduce in this section the spinal decomposition of the branching random walk. This result consists in an alternative description of a size-biased version of the law of the branching random walk. Spinal decomposition of a branching process has been introduced for the first time to study Galton-Watson processes in \cite{LPP95}. In \cite{Lyo97}, this technique is adapted to the study of branching random walks.

We precise in a first time the branching random walk notation we use. In particular, we introduce the Ulam-Harris notation for (plane) trees, and define the set of marked trees and marked trees with spine. Using this notation, we then state the spinal decomposition in Section~\ref{subsec:spinaldecomposition}, and use it in Section~\ref{subsec:application} to compute the number of individuals satisfying given properties.

\subsection{Branching random walk notation}
\label{subsec:notationbrw}

We denote by $\calU = \bigcup_{n \geq 0} \N^n$, with the convention $\N^0 = \{ \emptyset \}$ the set of finite sequences of positive integers. We write $\calU^* = \calU \backslash \{ \emptyset \}$. For any $u \in \calU$, we denote by $|u|$ the length of $u$. For $k \leq |u|$, we write $u(k)$ for the value of the $k$th element in $u$ and $u_k=(u(1),u(2),\ldots, u(k))$ for the restriction of $u$ to its $k$ first elements. We introduce $\pi : u \in \calU^* \mapsto u_{|u|-1} \in \calU$, and we call $\pi u$ the parent of $u$. Given $u,v \in \calU$, we write $u.v = (u(1),\ldots u(|u|),v(1),\ldots, v(|v|))$ for the concatenation of $u$ and $v$. We define a partial order on $\calU$ by
\[
  u < v \iff |u| < |v| \text{ and } v_{|u|} = u.
\]
We write $u \wedge v = u_{\min(|u|,|v|)} = v_{\min(|u|,|v|)}$ the most recent common ancestor of $u$ and $v$.

The set $\calU$ is used to encode any individual in a genealogical tree. We understand $u \in \calU$ as the $u(|u|)$th child of the $u(|u|-1)$th child of the ... of the $u(1)$th child of the initial ancestor $\emptyset$. A tree $\T$ is a subset of $\calU$ satisfying the three following assumptions:
\begin{equation*}
  \emptyset \in \T \text{,  if } u \in \T\text{, then } \pi u \in \T \text{, and if } u \in \T, \text{ then for any } j \leq u(|u|), \text{ } (\pi u).j \in \T.
\end{equation*}
Given a tree $\T$, the set $\{ u \in \T : |u|=n \}$ is called the $n$th generation of the tree, that we abbreviate into $\{ |u| = n\}$ if the tree we consider is clear in the context. For any $u \in \T$, we write $\Omega(u) = \{ v \in \T : \pi v = u\}$ for the set of children of $u$. We say that $\T$ has infinite height if for any $n \in \N$, $\{|u|=n\}$ is non-empty.

\begin{figure}[ht]
\centering
\subfigure[Ulam-Harris notation]{
\begin{tikzpicture}[scale=0.95]
  \node [individu] {$\emptyset$} [grow'=down, level distance = 1.2cm, sibling distance = 2cm]
  child { node [individu] {$3$} [sibling distance = 1.3cm]
    child { node [individu] {$33$} }
    child { node [individu] {$32$} [sibling distance = 1cm]
      child { node [individu] {$322$} }
      child { node [individu] {$321$} }
    }
    child { node [individu] {$31$} }
  }
  child { node [individu] {$2$} }
  child { node [individu] {$1$} [sibling distance = 1.5cm]
    child { node [individu] {$12$} [sibling distance = 1cm]
      child { node [individu] {$122$} }
      child { node [individu] {$121$} }
    }
    child { node [individu] {$11$} 
      child { node [individu] {$111$} }
      }
  };
  \node[draw,rectangle,rounded corners=3pt] [thick, color=purple] (P) at (3.5,-1.2) {$\pi(31)$};
  \draw[->,>=latex, thick, color=purple] (P) to (2.2,-1.2);
  \draw[thick, color=blue, rounded corners] (-2.2,-3.95) rectangle (-0.3,-3.25) ;
  \draw[thick, color=blue] (-1,-3.9) node[below] {$\Omega(21)$} ;
\end{tikzpicture}}%
\subfigure[Notation for the genealogy]{
\begin{tikzpicture}[scale=0.95]
  \draw[thick,color=blue!50] (0,0) node {$\bullet$} -- (0,-1.2) node {$\bullet$} ;
  \draw[thick,color=blue!50] (0,0) -- (-2,-1.2) node {$\bullet$} ;
  \draw[thick,color=blue!50] (0,0) -- (2,-1.2) node {$\bullet$} ;
  \draw[thick,color=blue!50] (-2,-1.2) -- (-2.75,-2.4) node {$\bullet$} ;
  \draw[thick,color=blue!50] (-2,-1.2) -- (-1.25,-2.4) node {$\bullet$} ;
  \draw[thick,color=blue!50] (-2.75,-2.4) -- (-2.75,-3.6) node {$\bullet$} ;
  \draw[thick,color=blue!50] (-1.25,-2.4) -- (-1.75,-3.6) node {$\bullet$} ;
  \draw[thick,color=blue!50] (-1.25,-2.4) -- (-0.75,-3.6) node {$\bullet$} ;
  \draw[thick,color=blue!50] (2,-1.2) -- (3.3,-2.4) node {$\bullet$} ;
  \draw[thick,color=blue!50] (2,-1.2) -- (2,-2.4) node {$\bullet$} ;
  \draw[thick,color=blue!50] (2,-1.2) -- (0.7,-2.4) node {$\bullet$} ;
  \draw[thick,color=blue!50] (2,-2.4) -- (1.5,-3.6) node {$\bullet$} ;
  \draw[thick,color=blue!50] (2,-2.4) -- (2.5,-3.6) node {$\bullet$} ;
  \draw[thick, color=white] (-1,-3.9) node[below] {$\Omega(21)$} ;
  \draw[thick] (-2.75,-3.6) node[left] {$u$};
  \draw[thick] (-0.75,-3.6) node[right] {$v$};
  \draw[thick] (-2.75,-2.4) node[left] {$u_2$};
  \draw[thick] (-2,-1.2) node[left] {$u_1 = u\wedge v$};
  \draw[thick] (0,0) node[left] {$u_0$} ;
  \draw[thick,color=blue] (0,0) node {$\bullet$} ;
  \draw[thick,color=blue] (0,0) -- (-2,-1.2) node {$\bullet$} ;
  \draw[thick,color=blue] (-2,-1.2) -- (-2.75,-2.4) node {$\bullet$} ;
  \draw[thick,color=blue] (-2,-1.2) -- (-1.25,-2.4) node {$\bullet$} ;
  \draw[thick,color=blue] (-2.75,-2.4) -- (-2.75,-3.6) node {$\bullet$} ;
  \draw[thick,color=blue] (-1.25,-2.4) -- (-0.75,-3.6) node {$\bullet$} ;
\end{tikzpicture}}
\caption{A plane tree}
\end{figure}
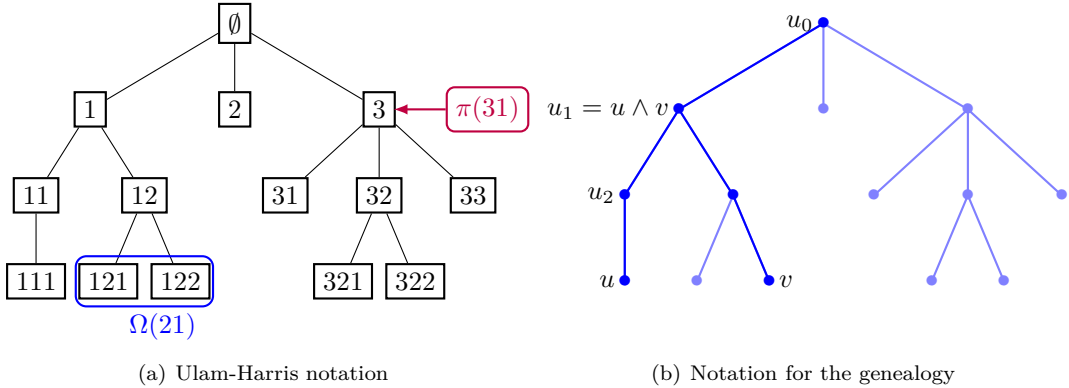

\begin{example}
Given a set $\{ \xi(u), u \in \calU \}$ of i.i.d. non-negative integer-valued random variables, we write $\T = \left\{ u \in \T : \forall k < |u|, u(k+1) \leq \xi(u_{k}) \right\}$, which is a Galton-Watson tree. It is well-known (see e.g. \cite{AtN04}) that $\T$ has infinite height with positive probability if and only if $\E(\xi(\emptyset))>1$.
\end{example}

A marked tree is a couple $(\T,V)$ such that $\T$ is a tree and $V : \T \mapsto \R$.  We refer to $V(u)$ as the position of $u$. The set of all marked trees is written $\mathcal{T}$. We introduce the filtration
\[
  \forall n \geq 0, \calF_n = \sigma\left( (u, V(u)), |u| \leq n \right).
\]
Given a marked tree $(\T,V)$ and $u \in \T$, we set $\T^u = \{ v \in \calU : u.v \in \T\}$. Observe that $\T^u$ is a tree. Moreover, writing $V^u : v \in \T^u \mapsto V(v.u)-V(u)$, we note that $(\T^u,V^u)$ is a marked tree called the subtree rooted at $u$ of $\T$.

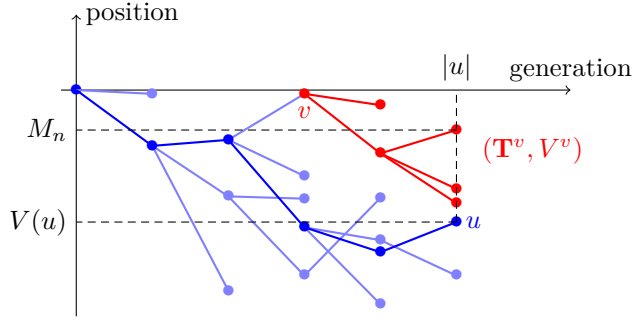
\begin{figure}[ht]
\centering
\begin{tikzpicture}
\draw [->] (-0.2,0) -- (6.5,0) node[above] {generation} ;
\draw [<-] (0,1)  node[right] {position} -- (0,-3) ;

\draw [thick, color=blue] (0,0) node {$\bullet$} ;
\draw[thick, color=blue!50] (0,-0.0) -- (1,-0.05) ;
\draw[thick, color=blue!50] (1,-0.05) node{$\bullet$};
\draw[thick, color=blue!50] (1,-0.74) -- (2,-2.66) ;
\draw[thick, color=blue!50] (2,-2.66) node{$\bullet$};
\draw[thick, color=blue!50] (1,-0.74) -- (2,-1.41) ;
\draw[thick, color=blue!50] (2,-1.41) node{$\bullet$};
\draw[thick, color=blue!50] (2,-0.66) -- (3,-0.05) ;
\draw[thick, color=red] (3,-0.05) node{$\bullet$};
\draw[red] (3,-0.1) node[below] {$v$};
\draw[thick, color=blue!50] (2,-0.66) -- (3,-1.14) ;
\draw[thick, color=blue!50] (3,-1.14) node{$\bullet$};
\draw[thick, color=blue!50] (2,-1.41) -- (3,-2.46) ;
\draw[thick, color=blue!50] (3,-2.46) node{$\bullet$};
\draw[thick, color=blue!50] (2,-1.41) -- (3,-1.44) ;
\draw[thick, color=blue!50] (3,-1.44) node{$\bullet$};
\draw[thick, color=red] (3,-0.05) -- (4,-0.84) ;
\draw[thick, color=red] (4,-0.84) node{$\bullet$};
\draw[thick, color=red] (3,-0.05) -- (4,-0.2) ;
\draw[thick, color=red] (4,-0.2) node{$\bullet$};
\draw[thick, color=blue!50] (3,-1.82) -- (4,-1.99) ;
\draw[thick, color=blue!50] (4,-1.99) node{$\bullet$};
\draw[thick, color=blue!50] (3,-1.82) -- (4,-2.84) ;
\draw[thick, color=blue!50] (4,-2.84) node{$\bullet$};
\draw[thick, color=blue!50] (3,-2.46) -- (4,-1.43) ;
\draw[thick, color=blue!50] (4,-1.43) node{$\bullet$};
\draw[thick, color=red] (4,-0.84) -- (5,-1.31) ;
\draw[thick, color=red] (5,-1.31) node{$\bullet$};
\draw[thick, color=red] (4,-0.84) -- (5,-0.53) ;
\draw[thick, color=red] (5,-0.53) node{$\bullet$};
\draw[thick, color=red] (4,-0.84) -- (5,-1.5) ;
\draw[thick, color=red] (5,-1.5) node{$\bullet$};
\draw[thick, color=blue!50] (4,-1.99) -- (5,-2.46) ;
\draw[thick, color=blue!50] (5,-2.46) node{$\bullet$};
\draw[thick, color=blue] (4,-2.15) -- (5,-1.75) ;
\draw[thick, color=blue] (5,-1.75) node{$\bullet$};
\draw[thick, color=blue] (3,-1.82) -- (4,-2.15) ;
\draw[thick, color=blue] (4,-2.15) node{$\bullet$};
\draw[thick, color=blue] (2,-0.66) -- (3,-1.82) ;
\draw[thick, color=blue] (3,-1.82) node{$\bullet$};
\draw[thick, color=blue] (1,-0.74) -- (2,-0.66) ;
\draw[thick, color=blue] (2,-0.66) node{$\bullet$};
\draw[thick, color=blue] (0,-0.0) -- (1,-0.74) ;
\draw[thick, color=blue] (1,-0.74) node{$\bullet$};
\draw [densely dashed] (5,-1.75) -- (0,-1.75) node[left] {$V(u)$};
\draw [densely dashed] (5,-1.75) -- (5,0) node[above] {$|u|$};
\draw [color=blue] (5,-1.75) node [right] {$u$} ;
\draw [densely dashed] (5,-0.53) -- (0,-0.53) node[left] {$M_n$};
\draw [color=red] (6,-0.8) node {$(\T^v,V^v)$};
\end{tikzpicture}\caption{A plane marked tree}
\end{figure}

A branching random walk is a random marked tree $(\T,V)$ such that the family of point processes $\left\{ (V(v)-V(u), v \in \Omega(u)), u \in \T \right\}$ is independent and identically distributed. Observe that in a branching random walk, for any $u \in \T$, $(\T^u,V^u)$ is a branching random walk independent of $\calF_{|u|}$.

Given a tree $\T$ of infinite height, we say that $w \in \N^\N$ is a spine of $\T$ if for any $n \in \N$, $w_n \in \T$ (where $w_n$ is the finite sequence of the $n$ first elements in $w$). A triplet $(\T,V,w)$, where $\T$ is tree, $V : \T \to \R$ and $w$ is a spine of $\T$ is called a marked tree with spine, and the set of such objects is written $\hat{\calT}$. For any $k \geq 0$, we write $\Omega_k = \Omega(w_k)\backslash\{w_{k+1}\}$, the set of children of the ancestor at generation $k$ of the spine, except its spine child. We define three filtrations on this set
\begin{equation}
  \label{eqn:filtdef}
  \calF_n = \sigma\left( (u, V(u)), |u| \leq n\right)
\text{ and } \calG_n = \sigma\left(w_k, V(w_k) : k \leq n \right) \vee \sigma\left(u,V(u), u \in \Omega_k, k < n\right).
\end{equation}
The filtration $\calF$ is obtained by ``forgetting'' about the spine of the process, while The filtration $\calG$ correspond to the knowledge of the position of the spine and its direct children.

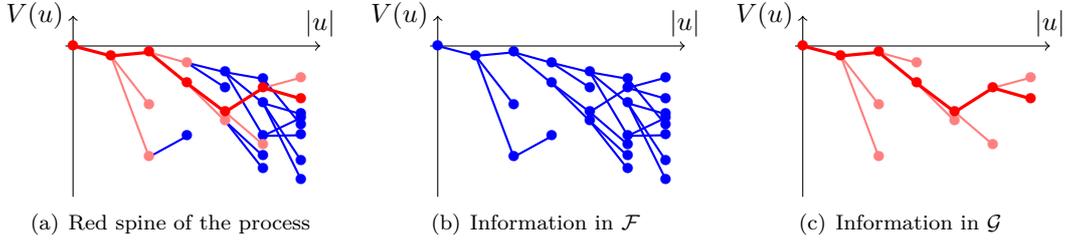
\begin{figure}[ht]
\centering
\subfigure[Red spine of the process]{
\begin{tikzpicture}[yscale=0.4,xscale=0.5]
\draw [->] (-0.2,0) -- (6.5,0) node[above] {$|u|$} ;
\draw [<-] (0,1) node[left] {$V(u)$}  -- (0,-5);
\draw[thick, color=blue] (2,-3.68) -- (3,-2.99) ;
\draw[thick, color=blue] (3,-2.99) node{$\bullet$};
\draw[thick, color=blue] (3,-0.56) -- (4,-0.87) ;
\draw[thick, color=blue] (4,-0.87) node{$\bullet$};
\draw[thick, color=blue] (3,-0.56) -- (4,-1.38) ;
\draw[thick, color=blue] (4,-1.38) node{$\bullet$};
\draw[thick, color=blue] (4,-0.87) -- (5,-1.9) ;
\draw[thick, color=blue] (5,-1.9) node{$\bullet$};
\draw[thick, color=blue] (4,-0.87) -- (5,-2.98) ;
\draw[thick, color=blue] (5,-2.98) node{$\bullet$};
\draw[thick, color=blue] (4,-0.87) -- (5,-1.09) ;
\draw[thick, color=blue] (5,-1.09) node{$\bullet$};
\draw[thick, color=blue] (4,-2.49) -- (5,-4.08) ;
\draw[thick, color=blue] (5,-4.08) node{$\bullet$};
\draw[thick, color=blue] (4,-2.49) -- (5,-3.65) ;
\draw[thick, color=blue] (5,-3.65) node{$\bullet$};
\draw[thick, color=blue] (5,-1.9) -- (6,-3.8) ;
\draw[thick, color=blue] (6,-3.8) node{$\bullet$};
\draw[thick, color=blue] (5,-1.9) -- (6,-2.26) ;
\draw[thick, color=blue] (6,-2.26) node{$\bullet$};
\draw[thick, color=blue] (5,-1.9) -- (6,-4.43) ;
\draw[thick, color=blue] (6,-4.43) node{$\bullet$};
\draw[thick, color=blue] (5,-2.98) -- (6,-2.96) ;
\draw[thick, color=blue] (6,-2.96) node{$\bullet$};
\draw[thick, color=blue] (5,-2.98) -- (6,-2.38) ;
\draw[thick, color=blue] (6,-2.38) node{$\bullet$};
\draw[thick, color=blue] (5,-1.09) -- (6,-2.62) ;
\draw[thick, color=blue] (6,-2.62) node{$\bullet$};
\draw[thick, color=red!50] (1,-0.33) -- (2,-1.94) ;
\draw[thick, color=red!50] (2,-1.94) node{$\bullet$};
\draw[thick, color=red!50] (1,-0.33) -- (2,-3.68) ;
\draw[thick, color=red!50] (2,-3.68) node{$\bullet$};
\draw[thick, color=red!50] (2,-0.21) -- (3,-0.56) ;
\draw[thick, color=red!50] (3,-0.56) node{$\bullet$};
\draw[thick, color=red!50] (3,-1.24) -- (4,-2.49) ;
\draw[thick, color=red!50] (4,-2.49) node{$\bullet$};
\draw[thick, color=red!50] (4,-2.19) -- (5,-3.27) ;
\draw[thick, color=red!50] (5,-3.27) node{$\bullet$};
\draw[thick, color=red!50] (5,-1.39) -- (6,-1.05) node{$\bullet$};
\draw[very thick, color=red] (0,0) node{$\bullet$} -- (1,-0.33) node{$\bullet$} -- (2,-0.21) node{$\bullet$}  -- (3,-1.24)node{$\bullet$} -- (4,-2.19)node{$\bullet$} -- (5,-1.39)node{$\bullet$} -- (6,-1.75)node{$\bullet$} ;

\end{tikzpicture}
}%
\subfigure[Information in $\mathcal{F}$]{
\begin{tikzpicture}[yscale=0.4,xscale=0.5]
\draw [->] (-0.2,0) -- (6.5,0) node[above] {$|u|$} ;
\draw [<-] (0,1) node[left] {$V(u)$}  -- (0,-5);
\draw[thick, color=blue] (0,0) node{$\bullet$} ;
\draw[thick, color=blue] (0,0.0) -- (1,-0.33) ;
\draw[thick, color=blue] (1,-0.33) node{$\bullet$};
\draw[thick, color=blue] (1,-0.33) -- (2,-0.21) ;
\draw[thick, color=blue] (2,-0.21) node{$\bullet$};
\draw[thick, color=blue] (1,-0.33) -- (2,-1.94) ;
\draw[thick, color=blue] (2,-1.94) node{$\bullet$};
\draw[thick, color=blue] (1,-0.33) -- (2,-3.68) ;
\draw[thick, color=blue] (2,-3.68) node{$\bullet$};
\draw[thick, color=blue] (2,-0.21) -- (3,-0.56) ;
\draw[thick, color=blue] (3,-0.56) node{$\bullet$};
\draw[thick, color=blue] (2,-0.21) -- (3,-1.24) ;
\draw[thick, color=blue] (3,-1.24) node{$\bullet$};
\draw[thick, color=blue] (2,-3.68) -- (3,-2.99) ;
\draw[thick, color=blue] (3,-2.99) node{$\bullet$};
\draw[thick, color=blue] (3,-0.56) -- (4,-0.87) ;
\draw[thick, color=blue] (4,-0.87) node{$\bullet$};
\draw[thick, color=blue] (3,-0.56) -- (4,-1.38) ;
\draw[thick, color=blue] (4,-1.38) node{$\bullet$};
\draw[thick, color=blue] (3,-1.24) -- (4,-2.49) ;
\draw[thick, color=blue] (4,-2.49) node{$\bullet$};
\draw[thick, color=blue] (3,-1.24) -- (4,-2.19) ;
\draw[thick, color=blue] (4,-2.19) node{$\bullet$};
\draw[thick, color=blue] (4,-0.87) -- (5,-1.9) ;
\draw[thick, color=blue] (5,-1.9) node{$\bullet$};
\draw[thick, color=blue] (4,-0.87) -- (5,-2.98) ;
\draw[thick, color=blue] (5,-2.98) node{$\bullet$};
\draw[thick, color=blue] (4,-0.87) -- (5,-1.09) ;
\draw[thick, color=blue] (5,-1.09) node{$\bullet$};
\draw[thick, color=blue] (4,-2.19) -- (5,-3.27) ;
\draw[thick, color=blue] (5,-3.27) node{$\bullet$};
\draw[thick, color=blue] (4,-2.49) -- (5,-4.08) ;
\draw[thick, color=blue] (5,-4.08) node{$\bullet$};
\draw[thick, color=blue] (4,-2.49) -- (5,-3.65) ;
\draw[thick, color=blue] (5,-3.65) node{$\bullet$};
\draw[thick, color=blue] (4,-2.19) -- (5,-1.39) ;
\draw[thick, color=blue] (5,-1.39) node{$\bullet$};
\draw[thick, color=blue] (5,-1.9) -- (6,-3.8) ;
\draw[thick, color=blue] (6,-3.8) node{$\bullet$};
\draw[thick, color=blue] (5,-1.9) -- (6,-2.26) ;
\draw[thick, color=blue] (6,-2.26) node{$\bullet$};
\draw[thick, color=blue] (5,-1.9) -- (6,-4.43) ;
\draw[thick, color=blue] (6,-4.43) node{$\bullet$};
\draw[thick, color=blue] (5,-2.98) -- (6,-2.96) ;
\draw[thick, color=blue] (6,-2.96) node{$\bullet$};
\draw[thick, color=blue] (5,-2.98) -- (6,-2.38) ;
\draw[thick, color=blue] (6,-2.38) node{$\bullet$};
\draw[thick, color=blue] (5,-1.09) -- (6,-2.62) ;
\draw[thick, color=blue] (6,-2.62) node{$\bullet$};
\draw[thick, color=blue] (5,-1.39) -- (6,-1.05) node{$\bullet$};
\draw[thick, color=blue] (5,-1.39) -- (6,-1.75) node{$\bullet$};
\end{tikzpicture}
}%
\subfigure[Information in $\mathcal{G}$]{
\begin{tikzpicture}[yscale=0.4,xscale=0.5]
\draw [->] (-0.2,0) -- (6.5,0) node[above] {$|u|$} ;
\draw [<-] (0,1) node[left] {$V(u)$}  -- (0,-5);
\draw[thick, color=red!50] (1,-0.33) -- (2,-1.94) ;
\draw[thick, color=red!50] (2,-1.94) node{$\bullet$};
\draw[thick, color=red!50] (1,-0.33) -- (2,-3.68) ;
\draw[thick, color=red!50] (2,-3.68) node{$\bullet$};
\draw[thick, color=red!50] (2,-0.21) -- (3,-0.56) ;
\draw[thick, color=red!50] (3,-0.56) node{$\bullet$};
\draw[thick, color=red!50] (3,-1.24) -- (4,-2.49) ;
\draw[thick, color=red!50] (4,-2.49) node{$\bullet$};
\draw[thick, color=red!50] (4,-2.19) -- (5,-3.27) ;
\draw[thick, color=red!50] (5,-3.27) node{$\bullet$};
\draw[thick, color=red!50] (5,-1.39) -- (6,-1.05) node{$\bullet$};
\draw[very thick, color=red] (0,0) node{$\bullet$} -- (1,-0.33) node{$\bullet$} -- (2,-0.21) node{$\bullet$}  -- (3,-1.24)node{$\bullet$} -- (4,-2.19)node{$\bullet$} -- (5,-1.39)node{$\bullet$} -- (6,-1.75)node{$\bullet$} ;
\end{tikzpicture}
}
\end{figure}

\subsection{The spinal decomposition}
\label{subsec:spinaldecomposition}

Let $(\T,V)$ be a branching random walk satisfying \eqref{eqn:supercritical} and \eqref{eqn:boundary}. For any $x \in \R$, we write $\P_x$ for the law of $(\T,V+x)$ and $\E_x$ for the corresponding expectation. For any $n \geq 0$, we set $W_n = \sum_{|u|=n} e^{V(u)}$. We observe that $(W_n)$ is a non-negative $(\calF_n)$-martingale. We define the law
\begin{equation}
  \label{eqn:sizeBiaisedLaw}
  \left. \bar{\P}_x\right|_{\calF_n} = e^{-x} W_n \cdot \left.\P_x \right|_{\calF_n}.
\end{equation}
The spinal decomposition consists in an alternative construction of the law $\bar{\P}_x$, as the projection of a law on $\hat{\calT}$, which we now define.

We write $\calL$ (respectively $\hat{\calL}$) for the law of the point process $(V(u), |u|=1)$ under the law $\P$ (resp. $\bar{\P}$). We observe that $\hat{\calL} = \sum_{|u|=1} e^{V(u)} \cdot \calL$. Let $(\hat{L}_n, n \in \N)$ be i.i.d. point processes with law $\hat{\calL}$. Conditionally on this sequence, we choose independently at random, for every $n \in \N$, $w(n) \in \N$ such that writing $\hat{L}_n = (\ell_n(1),\ldots \ell_n(N_n))$ we have
\[
  \forall h \in \N, \text{ } \P\left( w(n) = h \left| (\hat{L}_n, n \in \N) \right. \right) = \ind{h \leq N_n}\frac{e^{\ell_n(h)}}{\sum_{j \leq N_n} e^{\ell_n(j)}}.
\]
We write $w$ for the sequence $(w(n), n \in \N)$.

We introduce a family of independent point processes $\left\{L^u, u \in \calU\right\}$ such that $L^{w_k} = \hat{L}_{k+1}$, and if $u \neq w_{|u|}$, then $L^u$ has law $\calL$. For any $u \in \calU$ such that $|u| \leq n$, we write $L^u = (\ell^u_1,\ldots \ell^u_{N(u)})$. We construct the random tree $\T = \left\{ u \in \calU : |u| \leq n, \forall 1 \leq k \leq |u|, u(k) \leq N(u_{k-1})\right\}$,
and the the function $V : u \in \T \mapsto  \sum_{k=1}^{|u|} \ell^{u_{k-1}}_{u(k)}$. For all $x \in \R$, the law of $(\T, x + V, w) \in \hat{\calT}_n$ is written $\hat{\P}_x$, and the corresponding expectation is $\hat{\E}_x$. This law is called the law of the branching random walk with spine.

We can describe the branching random walk with spine as a process in the following manner. It starts with a unique individual positioned at $x \in \R$ at time 0, which is the ancestral spine $w_0$. Then, at each time $n \in \N$, every individual alive at generation $n$ dies. Each of these individuals gives birth to children, which are positioned around their parent according to an independent point process. If the parent is $w_n$, the law of this point process is $\hat{\calL}$, otherwise the law is $\calL$. The individual $w_{n+1}$ is then chosen at random among the children $u$ of $w_n$, with probability proportional to $e^{V(u)}$. Observe that under the law $\hat{\P}$,
\[
  \big(V(w_{n+1})-V(w_{n}), n \geq 0\big), \, \big((V(u)-V(w_n), u \in \Omega_n), n \geq 0\big), \text{ and } \big\{(\T^u,V^u), u \in \cup_{n \in \N} \Omega_n \big\},
\]
are respectively i.i.d. random variables, i.i.d. point processes and i.i.d. branching random walks. Note that while the branching random walks are independent of $\calG$, $V(w_{n+1})-V(w_{n})$ and $(V(u)-V(w_n), u \in \Omega_n)$ might be correlated.

\begin{figure}[ht]
\centering
\begin{tikzpicture}[xscale=0.7,yscale=0.45]
	\draw [color=red!50] (0,0) -- (8,-2) ;
	\draw [color=red!50] (0,0) -- (6,-2) ;
	\draw [color=red!50] (0,0) -- (-8,-2) ;
	\draw [very thick, color=red] (0,0) node {$\bullet$} ;
	\draw (0,0) node[above] {$w_0$} ;
	\draw [very thick, color=red] (0,0) -- (-3,-2) ;
	\triangleBRW{-8}{1}
	\triangleBRW{6}{1}
	\triangleBRW{8}{1}
	
	\draw [color=red!50] (-3,-2) -- (-6,-4) ;
	\triangleBRW{-6}{2}
	\draw [color=red!50] (-3,-2) -- (-4,-4) ;
	\triangleBRW{-4}{2}
	\draw [very thick, color=red] (-3,-2) node{$\bullet$} ;
	\draw (-3,-2) node[above left] {$w_1$} ;
	\draw [very thick, color=red] (-3,-2) -- (2,-4) ;
	
	\draw [color=red!50](2,-4) -- (2,-6);
	\triangleBRW{2}{3}
	\draw [color=red!50](2,-4) -- (4,-6) ;
	\triangleBRW{4}{3} ;
	\draw [very thick, color=red] (2,-4) node{$\bullet$} ;
	\draw (2,-4) node[above right] {$w_2$} ;
	\draw [very thick, color=red] (2,-4) -- (-1,-6);
	
	\draw (-1,-6) node[above left]{$w_3$} ;
	\draw [densely dashed, very thick, color=red] (-1,-6) -- (-2.5,-8) ;
	\draw [densely dashed,color=red!50] (-1,-6) -- (-1.5,-8) ;
	\draw [densely dashed,color=red!50] (-1,-6) -- (-0.5,-8) ;
	\draw [densely dashed,color=red!50] (-1,-6) -- (0.5,-8) ;
	\draw [very thick, color=red] (-1,-6) node{$\bullet$} ;
\end{tikzpicture}
\caption{The tree of the branching random walk with spine}
\end{figure}

The following result links the laws $\hat{\P}_x$ and $\bar{\P}_x$. The spinal decomposition is proved in \cite{Lyo97}.
\begin{proposition}[Spinal decomposition]
\label{prop:spinaldecomposition}
Assuming \eqref{eqn:supercritical} and \eqref{eqn:boundary}, for any $n \in \N, x \in \R$ and $|u|=n$, we have
\[
  \bar{\P}_x\big|_{\calF_n} = \hat{\P}_x\big|_{\calF_n} \quad \text{ and } \quad \hat{\P}_x(w_n = u |\calF_n) = \frac{e^{V(u)}}{W_n}.
\]
\end{proposition}

The spinal decomposition enables to compute moments of any additive functional of the branching random walk, and links it to random walk estimates, as we observe in the next section.

\subsection{Application of the spinal decomposition}
\label{subsec:application}

Let $(\T,V)$ be a branching random walk. For any $n \in \N$, we denote by $\mathcal{A}_n$ a subset of $\{|u|=n\}$, such that $\{ u \in \mathcal{A}_n \} \in \mathcal{F}_n$. We write $A_n$ for the number of elements in $\mathcal{A}_n$. We compute in this section the first and second moments of this random variable.

\begin{lemma}
\label{lem:firstmoment}
For any $n \geq 0$ and $x \in \R$, we have
\[\E_x(A_n) = e^x\hat{\E}_x\left(e^{-V(w_n)} \ind{w_n \in \mathcal{A}_n} \right).\]
\end{lemma}

\begin{proof}
Applying the spinal decomposition, we observe that
\begin{align*}
  \E_x(A_n) &= e^x\bar{\E}_x\left( \frac{1}{W_n} \sum_{|u|=n} \ind{u \in \mathcal{A}_n} \right)= e^x\hat{\E}_x\left( \sum_{|u|=n} \frac{e^{V(u)}}{W_n} e^{-V(u)} \ind{u \in \mathcal{A}_n} \right)\\
  &= e^x\hat{\E}_x\left( \sum_{|u|=n} \hat{\P}(u=w_n|\calF_n) e^{-V(u)} \ind{u \in \mathcal{A}_n}  \right)= e^x\hat{\E}_x\left( e^{-V(w_n)} \ind{w_n \in \mathcal{A}_n} \right).
\end{align*}
\end{proof}

An immediate consequence of this result is the celebrated many-to-one lemma. This equation, known at least from the early works of Kahane and Peyrière \cite{Pey74,KaP76}, enables to compute the mean of any additive functional of the branching random walk.
\begin{lemma}[Many-to-one lemma]
\label{lem:manytoone}
There exists a random walk $S$ verifying $\P_x(S_0=x)=1$ such that for any measurable positive function $f$, we have
\[
  \E_x\left( \sum_{|u|=n} f(V(u_j), j \leq n) \right) = \E_x\left( e^{x-S_n}f(S_j, j \leq n) \right).
\]
\end{lemma}

\begin{proof}
Let $B$ be a measurable subset of $\R^n$, we set $\mathcal{A}_n = \left\{ |u|=n : (V(u_j), j \leq n) \in B \right\}$. Applying Lemma~\ref{lem:firstmoment} yields
\[  \E_x\left( \sum_{|u|=n} \ind{(V(u_j), j \leq n) \in B} \right) = \hat{\E}_x\left( e^{x-V(w_n)} \ind{(V(w_j), j\leq n) \in B} \right)
  =\E_x\left( e^{x-S_n} \ind{(S_j, j \leq n) \in B} \right),\]
where $S$ is a random walk under law $\P_x$ with the same law as $(V(w_j), j \geq 0)$ under law $\hat{\P}_x$. This equality being true for any measurable set $B$, it is true for any measurable positive functions, taking increasing limits.
\end{proof}

Similarly, we can compute the second moment of $A_n$.
\begin{lemma}
\label{lem:secondmoment}
For any $0 \leq k < n$, we write $A^{(2)}_{n,k} = \sum_{|u|=n, |v|=n} \ind{|u\wedge v|=k} \ind{u \in \mathcal{A}_n} \ind{v \in \mathcal{A}_n}$. We have $A_n^2 = A_n +\sum_{k=0}^{n-1} A^{(2)}_{n,k}$. Moreover, for any $0 \leq k < n$,
\[
  \E(A^{(2)}_{n,k}) = \hat{\E}\left( e^{-V(w_n)} \ind{w_n \in \mathcal{A}_n} B_{k,n} \right),
\]
where $B_{k,n} = \sum_{u \in \Omega_k} \hat{\E}\left( \left. \sum_{|v|=n,v \geq u} \ind{v \in \mathcal{A}_n} \right| \calG_n \right)$.
\end{lemma}

\begin{proof}
The first equality is immediate: $A_n^2$ is the number of couples of individuals in $\mathcal{A}_n$, and we may partition this set according to the generation at which the most recent common ancestor was alive. We now use the spinal decomposition in the same way as in Lemma~\ref{lem:firstmoment}, we have
\begin{align*}
  \E(A^{(2)}_{n,k}) &= \hat{\E} \left( \sum_{|u|=n} \frac{e^{V(u)}}{W_n} e^{-V(u)} \ind{u \in \mathcal{A}_n} \sum_{|v|=n} \ind{|u \wedge v|=k, v \in \mathcal{A}_n} \right)\\
  &= \hat{\E} \left( e^{-V(w_n)} \ind{w_n \in \mathcal{A}_n} H_{k,n} \right),
\end{align*}
where $H_{k,n} = \sum_{|v|=n} \ind{|v \wedge w_n|=k} \ind{v \in \mathcal{A}_n} = \sum_{u \in \Omega_k} \sum_{|v|=n, v \geq u} \ind{v \in \mathcal{A}_n}$. Thus, conditioning on $\mathcal{G}_n$, we have
\[
  \E(A_{n,k}^{(2)}) = \hat{\E} \left( e^{-V(w_n)} \ind{w_n \in \mathcal{A}_n} B_{k,n} \right).
\]
\end{proof}

\section{Some random walk estimates}
\label{sec:randomwalk}

We collect in this section a series of random walk estimates, and extend these results to random walk enriched with additional random variables, correlated to the last step of the walk. More precisely, let $(X_j,\xi_j)$ be i.i.d. random variables on $\R^2$ such that
\begin{equation}
  \label{eqn:integrabilityrw}
  \E(X_1) = 0, \quad \sigma^2 := \E(X_1^2) \in (0,+\infty) \quad \text{and} \quad \E((\xi_1)_+^2)< +\infty.
\end{equation}
We denote by $T_n = T_0 + \sum_{j=1}^n X_j$, where for any $x \in \R$, $\P_x(T_0=x) = 1$. We write $\P=\P_0$ for short. The process of interest $(T_{n-1}, \xi_n), n \geq 1)$ is an useful toy-model to study the spine of the branching random walk. It has similar structure as  $\Big( \big( V(w_{n}), (V(u)-V(w_n), u \in \Omega_n) \big), n \geq 0 \Big)$.

We first recall bounds on the value taken by $T_n$ at time $n$. Stone's local limit theorem \cite{Sto65} bounds the probability for a random walk to end up in an interval of finite size. For any $a,h>0$,
\begin{equation}
  \label{eqn:locallimitub}
   \limsup_{n \to +\infty} n^{1/2}\sup_{|y| \geq a n^{1/2}} \P(T_n \in [y,y+h]) \leq C (1+h)e^{-\frac{a^2}{2\sigma^2}}.
\end{equation}
Caravenna and Chaumont obtained in \cite{CaC08} a similar result for a random walk conditioned to stay non-negative. Given $(r_n) = O(n^{1/2})$, there exists $H>0$ such that for any $0 < a < b$,
\begin{equation}
  \label{eqn:locallimitlb}
  \liminf_{n \to +\infty} n^{1/2}\inf_{y \in [0,r_n]} \inf_{x \in [an^{1/2},bn^{1/2}]} \P(T_n \in [x,x+H]|T_j \geq -y, j \leq n) > 0.
\end{equation}
In the rest of the paper, $H$ is a fixed constant, large enough such that \eqref{eqn:locallimitlb} holds for the random walk followed by the spine of the branching random walk.

We now study the probability for a random walk to stay above a given boundary $(f_n)$, that is $O(n^{1/2-\epsilon})$. This result is often called in the literature the ballot theorem (see e.g. \cite{ABR08} for a review on this type of results). The upper bound we use here can be found in \cite[Lemma 3.6]{Mal14a},
\begin{equation}
  \label{eqn:ballotub}
  \sup_{n \in \N, y \geq 0} \frac{n^{1/2}}{1+y} \P\left( T_j \geq f_j - y, j \leq n \right) < +\infty.
\end{equation}
The lower bound is a result of Kozlov \cite{Koz76},
\begin{equation}
  \label{eqn:ballotlb}
  \inf_{n \in \N, y \in [0,n^{1/2}]} \frac{n^{1/2}}{1+y} \P\left( T_j \geq - y, j \leq n \right) < +\infty.
\end{equation}

Using these results, we are able to compute the probability for a random walk to make an excursion above a given curve. The proofs of the following results are very similar to the proofs of \cite[Lemmas 4.1 and 4.3]{AiS10}, and are postponed to the Appendix A.1.
\begin{lemma}
\label{lem:excursionmarchealeatoireub}
Let $A > 0$ and $\alpha < 1/2$, for any $n \geq 1$ we consider a function $k \mapsto f_n(k)$ such that
\[
  f_n(0)=0 \quad \text{and} \quad \sup_{j \leq n} \max\left( \frac{|f_n(j)|}{j^{\alpha}}, \frac{|f_n(n)-f_n(j)|}{(n-j)^\alpha} \right) \leq A.
\]
There exists $C>0$ such that for all $y,z,h > 0$,
\begin{multline*}
  \P\left( T_n - f_n(n) + y \in [z-h,z], T_{j} \geq f_n(j) - y, j \leq n \right) \\
  \leq C \frac{(1+y \wedge n^{1/2})(1+h \wedge n^{1/2})(1+z \wedge n^{1/2})}{n^{3/2}}.
\end{multline*}
\end{lemma}

Similarly, we obtain a lower bound for this quantity.
\begin{lemma}
\label{lem:excursionmarchealeatoirelb}
Let $\lambda > 0$, for any $0 \leq k \leq n$ we write $f_n(k) = \lambda \log \frac{n-k+1}{n+1}$. There exists $c>0$ such that for any $y \in [0,n^{1/2}]$,
\[
  \P\left( T_n \leq f_n(n)-y+H, T_j \leq f_n(j)-y, j \leq n \right) \geq c \frac{1+y}{n^{3/2}}.
\]
\end{lemma}

Finally, we are able to bound the probability for the random walk to make an excursion, while the additional random variables remain small. This proof, very similar to \cite[Lemma B.2]{Aid13} is also postponed to Appendix A.2.
\begin{lemma}
\label{lem:excursionSpine}
With the same notation as in the previous lemma, we set
\[
  \tau = \inf\{ k \leq n : T_k \leq f_n(k)-y + \xi_{k+1} \}.
\]
There exists $C>0$ such that for all $y\geq 0$,
\begin{multline*}
  \qquad\qquad   \P\left( T_n \leq f_n(n)-y+H, \tau < n, T_j \geq f_n(j)-y, j \leq n \right) \\ 
  \leq C\frac{1+y}{n^{3/2}}\left(\P(\xi_{1} \geq 0) + \E\left((\xi_1)_+^2\right) \right)\qquad \qquad 
\end{multline*}
\end{lemma}

\section{Bounding the tail of the maximal displacement}
\label{sec:tailbehaviour}

Let $(\T,V)$ be a branching random walk satisfying \eqref{eqn:supercritical}, \eqref{eqn:boundary} and \eqref{eqn:spine}, and $M_n$ its maximal displacement at time $n$. We write $m_n = -\frac{3}{2} \log n$. The main result of the section is the following estimate on the left tail of $M_n$.
\begin{theorem}
\label{thm:tailestimate}
Assuming \eqref{eqn:supercritical}, \eqref{eqn:boundary} and \eqref{eqn:spine}, there exist $c,C>0$ such that for any $n \geq 1$ and $y \in [0,n^{1/2}]$, $c (1+y)e^{- y} \leq \P\left( M_n \geq m_n + y \right) \leq C (1  + y)e^{- y}$.
\end{theorem}

A natural way to compute an upper bound for $\P(M_n \geq m_n + y)$ would be a direct application of the Markov inequality. We have
\[
  \P(M_n \geq m_n + y) \leq \E\left( \sum_{|u|=n} \ind{V(u) \geq m_n + y}\right) \leq \E\left( e^{ S_n} \ind{S_n \geq m_n + y} \right),
\]
by Lemma~\ref{lem:manytoone}. Therefore, as $S_n$ is a centrer random walk, we have
\[
  \P(M_n \geq m_n + y) \leq e^{- m_n- y} \sum_{h=0}^{+\infty} e^{- h} \P(S_n -m_n - y \in [h,h+1]) \leq C \frac{n^{3/2} e^{- y}}{n^{1/2}},
\]
by \eqref{eqn:locallimitub}. But this computation is not precise enough to yield Theorem~\ref{thm:tailestimate}.

To obtain a more precise upper bound, instead of computing the number of individuals that are at time $n$ greater than $m_n$, we compute for any generation $k \leq n$ the number of individuals alive at generation $k$ such that the probability that one of their children is greater than $m_n$ is larger than a given threshold. If we assume Theorem~\ref{thm:tailestimate} to be true, for an individual $u$ alive at generation $k$, the median of the position of its largest descendant is close to $V(u) + m_{n-k}$. Therefore, we compute (see Figure 6) the number of individuals that cross for the first time at time $k \leq n$ the boundary $k \mapsto f_n(k) := \frac{3}{2} \log \frac{n-k+1}{n+1} \approx m_n - m_{n-k}$.

\begin{figure}[ht]
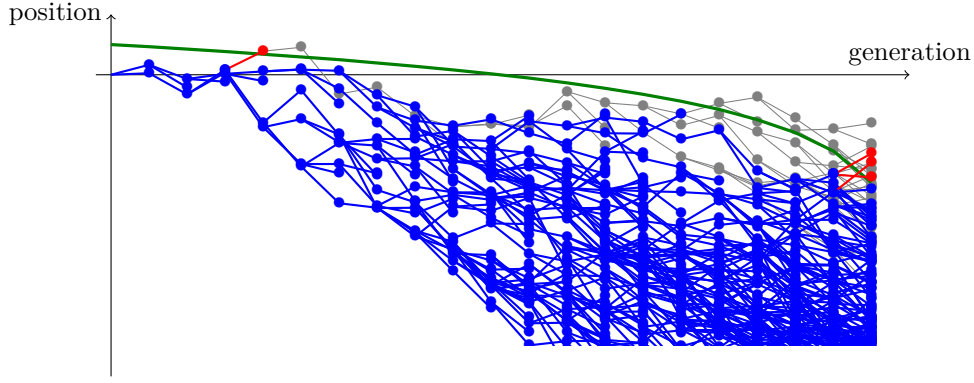

\centering
 
\caption{The boundary of the branching random walk}
\end{figure}

\begin{lemma}
\label{lem:upperfrontier}
Assuming \eqref{eqn:supercritical} and \eqref{eqn:boundary}, there exists $C>0$ such that for any $y \geq 0$,
\[
  \P\left( \exists |u| \leq n : V(u) \geq f_n(|u|)+y \right) \leq C (1+y)e^{-y}.
\]
\end{lemma}

\begin{proof}
For all $k \leq n$, we write $Z^{(n)}_k(y) = \sum_{|u|=k} \ind{V(u) \geq f_n(k) + y} \ind{V(u_j) < f_n(j) + y, j < k}$ the number of individuals crossing for the first time the curve $j \mapsto f^{(n)}_j$ at time $k$. By Lemma~\ref{lem:manytoone}, we have
\begin{align*}
  \E\left(Z^{(n)}_k(y) \right) &= \E\left[ e^{- S_k} \ind{S_k \geq f_n(k) +y} \ind{S_j \leq f_n(j) + y, j < k} \right]\\
  &\leq e^{-f_n(k) - y} \P\left(S_k \geq f_n(k) + y, S_j \leq f_n(j)+ y, j < k\right).
\end{align*}
We condition this probability with respect to the last step $S_k - S_{k-1}$ to obtain
\[
  \P\left(S_k \geq f_n(k) + y, S_j \leq f_n(j) + y, j < k\right)= \E(\phi_{k-1}(S_k - S_{k-1})).
\]
where $\phi_k(x) = \P\left(S_k \geq f_n(k) + y - x, S_j \leq f_n(j) + y, j \leq k\right)$. Applying Lemma~\ref{lem:excursionmarchealeatoireub}, there exists $C>0$ such that for all $k \leq n$ and $x \in \R$, $\phi_k(x) \leq C \ind{x \geq 0} (1 + y)(1 + x)^2(k+1)^{-3/2}$. Thus,
\begin{multline*}
  \P\left(\exists |u| \leq n : V(u) \geq f_n(|u|) + y\right) \leq \sum_{k=0}^n \P(Z^{(n)}_k(y) \geq 1) \leq \sum_{k=0}^n \E(Z^{(n)}_k(y))\\
  \leq C (1 + y) e^{- y} \sum_{k=0}^n \tfrac{(n+1)^{3/2}}{(k+1)^{3/2} (n-k+1)^{3/2}} \E\left((S_k - S_{k-1})_+^2 + 1\right).
\end{multline*}
As a consequence, we obtain
\begin{align*}
  \P(\exists |u| \leq n : V(u) \geq f_n(|u|) + y) &\leq C (1 + y) e^{-y} \left( \sum_{k=0}^{n/2} 2^{3/2} k^{-3/2} + \sum_{k=n/2}^n 2^{3/2}(n-k+1)^{-3/2} \right)\\
  &\leq C (1+y)e^{-y}.
\end{align*}
\end{proof}

This lemma directly implies the upper bound in Theorem~\ref{thm:tailestimate}.
\begin{proof}[Proof of the upper bound in Theorem~\ref{thm:tailestimate}]
As $f_n(n) = - \frac{3}{2} \log (n+1) \geq - \frac{3}{2} \log n - 2$, we observe that
\[
  \P\left( M_n \geq m_n + y \right) \leq \P\left(\exists |u| \leq n : V(u) \geq f_n(|u|) + y - 2\right).
\]
We apply Lemma~\ref{lem:upperfrontier} to conclude the proof.
\end{proof}

To obtain a lower bound for $\P(M_n \geq m_n + y)$, we apply the Cauchy-Schwarz inequality: if $X$ is an integer-valued non-negative random variable, we have $\P(X \geq 1) \geq \frac{\E(X)^2}{\E(X^2)}$. 

A good control on the second moment of the number of individuals staying below $f_n$ until time $n$ is obtained by considering only the individuals that do not make ``too many children too high''. To quantify this property, we set $\xi(u) = \log \sum_{v \in \Omega(u)} \left( 1 + (V(u)-V(v))_+ \right) e^{V(u) - V(v)}$ for any $u \in \T$. For $n \in \N$ and $y,z \geq 0$, we define the sets
\begin{align*}
  \mathcal{A}_n(y)& = \left\{ |u| \leq n : V(u_j) \leq f_n(j) + y, j \leq |u| \right\},\\
  \bar{\mathcal{A}}_n(y) &= \left\{ |u| = n : u \in \mathcal{A}_n(y), V(u) \geq f_n(n)+y-H \right\}\\
  \mathcal{B}_n(y,z)& = \left\{ |u| \leq n : \xi(u_j) \leq z + (f_n(j)+y-V(u_j))/2 , j \leq |u| \right\}.
\end{align*}
We write $\mathcal{Y}_n(y,z) = \bar{\mathcal{A}}_n(y) \cap \mathcal{B}_n(y,z)$, and take interest in $Y_n(y,z) = \#\mathcal{Y}_n(y,z)$, the number of individuals that made an excursion below $f_n$, while making not too many children. We also set
\[
  Y^{(2)}_{n,k}(y,z) = \sum_{|u|=|v|=n} \ind{|u\wedge v|=k} \ind{u \in \mathcal{Y}_n(y,z)} \ind{v \in \in \mathcal{Y}_n(y,z)},
\]
and we compute the mean of this quantity

\begin{lemma}
\label{lem:meanupperbound}
Assuming \eqref{eqn:supercritical} and \eqref{eqn:boundary}, there exists $C>0$ such that for $y,z \geq 0$ and $0 \leq k \leq n$,
\[
  \E(Y_n(y,z)) \leq C (1 + y)e^{-y} \quad \text{and} \quad
  \E\left( Y^{(2)}_{n,k}(y,z) \right) \leq C \frac{(n+1)^{3/2}}{(k+1)^{3/2}(n-k+1)^{3/2}} (1+y)e^{z - y}.
\]
In particular, $\E\left( Y_n(y,z)^2 \right) \leq C (1+y) e^{z-y}$.
\end{lemma}

\begin{proof}
We first apply Lemma~\ref{lem:firstmoment}, we have
\begin{align*}
  \E\left( Y_n(y,z) \right) &= \hat{\E}\left( e^{-V(w_n)} \ind{w_n \in \mathcal{Y}_n(y,z)} \right)
  \leq \hat{\E} \left( e^{-V(w_n)} \ind{w_n \in \bar{\mathcal{A}}_n(y)} \right)\\
  &\leq (n+1)^{3/2}e^{H-y} \hat{\P}\left(V(w_n) \geq f_n(n) + y - H, V(w_j) \leq f_n(j) + y, j \leq n \right)\\
  &\leq C(1+y)e^{-y} \quad \text{by Lemma~\ref{lem:excursionmarchealeatoireub}.}
\end{align*}

To bound the mean of $Y^{(2)}_{n,k}$, applying Lemma~\ref{lem:secondmoment} we can write
\begin{equation}
  \E\left( Y^{(2)}_{n,k}(y,z) \right) = \hat{\E}\left( e^{-V(w_n)} \ind{w_n \in \mathcal{Y}_n(y,z)} Z_{k,n} \right),
  \label{eqn:tomerge}
\end{equation}
where $Z_{k,n} = \sum_{u \in \Omega_k} \hat{\E}\left( \left. \sum_{|v|=n, v \geq u} \ind{v \in \mathcal{Y}_n(y,z)} \right| \mathcal{G}_n \right)$. We observe that by definition of $\hat{\P}$, for any $k < n$ and $u \in \Omega_k$, the subtree $(\T^u,V^u)$ is a branching random walk independent of $\mathcal{G}_n$. Therefore, applying Lemma~\ref{lem:firstmoment}, we have
\begin{align*}
  &\hat{\E}\left( \left. \sum_{|v|=n, v \geq u} \ind{v \in \mathcal{Y}_{n}(y,z)} \right| \mathcal{G}_n \right)\\
  \leq & \E_{V(u)} \left( \sum_{|v|=n-k-1} \ind{V(v) \geq f_n(n)+y-H, V(v_j) \leq f_n(j+k+1)+y, j \leq n-k-1} \right)\\
  \leq & e^{V(u)} \hat{\E}_{V(u)} \left( e^{-V(w_{n-k-1})} \ind{V(w_{n-k-1}) \geq f_n(n)+y-H, V(w_j) \leq f_n(j+k+1)+y, j \leq n-k-1}\right)\\
  \leq & (n+1)^{3/2} e^{V(u)-y} \hat{\P}_{V(u)} \left(\begin{array}{l}V(w_{n-k-1}) \geq f_n(n)+y-H,\\ V(w_k) \leq f_n(j+k+1)+y, j \leq n-k-1 \end{array}\right)\\
  \leq & C \frac{(n+1)^{3/2}}{(n-k+1)^{3/2}}(1+(f_n(k+1)+y-V(u))_+) e^{V(u)-y},
\end{align*}
by Lemma~\ref{lem:excursionmarchealeatoireub} again. As $x \mapsto x_+$ is 1-Lipschitz, we obtain
\begin{align*}
  Z_{k,n} &\leq C \frac{(n+1)^{3/2}}{(n-k+1)^{3/2}} (1 + (f_n(k)+y-V(w_k))_+)e^{V(w_k)-y} e^{\xi(w_k)}\\
  &\leq Ce^{\xi(w_k)} \phi(f_n(k)+y-V(w_k)),
\end{align*}
where $\phi : x \mapsto (1 + x_+)e^{-x}$. For $k < n$, \eqref{eqn:tomerge} becomes
\begin{align*}
  \E\left( Y^{(2)}_{n,k}(y,z) \right)
  \leq &C \hat{\E}\left( e^{\xi(w_k)-V(w_n)} \phi(f_n(k)+y-V(w_k)) \ind{w_n \in \mathcal{Y}_n(y,z)} \right)\\
  \leq &C (n+1)^{3/2}e^{z-y} \hat{\E}\left( \phi((f_n(k)+y-V(w_k))/2)\ind{w_n \in \bar{\mathcal{A}}_n(y)} \right),
\end{align*}
as $w_n \in \mathcal{B}_n(y,z)$. Conditioning on the value taken at time $j$ by $V(w_j)-f_n(j)-y$, we have
\begin{align*}
  &\hat{\E}\left(\phi((f_n(k)+y-V(w_k))/2)\ind{w_n \in \bar{\mathcal{A}}_n(y)} \right)\\
  \leq & C\sum_{h=0}^{+\infty} \phi((h+1)/2) \hat{\P}\left( w_n \in \bar{\mathcal{A}}_n(y), f_n(k)+y-V(w_k)\in [h,h+1] \right)\\
  \leq & C\sum_{h=0}^{+\infty} (h+2) e^{-h/2} C\frac{(1+y)(1+h)}{(k+1)^{3/2}} \frac{(1+h)}{(n-k+1)^{3/2}}\\
  \leq & C \frac{1+y}{(k+1)^{3/2}(n-k+1)^{3/2}},
\end{align*}
applying Lemma~\ref{lem:excursionmarchealeatoireub} on the intervals $[0,j]$ and $[j,n]$. We conclude that
\[
  \E\left( Y^{(2)}_{n,k}(y,z) \right) \leq C \frac{(n+1)^{3/2}}{(k+1)^{3/2}(n-k+1)^{3/2}} (1+y)e^{z-y}.
\]
Moreover, using again Lemma \ref{lem:secondmoment}, we have
\[
\E\left( Y_n(y,z)^2 \right) = \E(Y_n(y,z)) + \sum_{k=0}^{n-1} \E\left( Y^{(2)}_{n,k}(y,z) \right) \leq C (1+y)e^{z-y}
\]
\end{proof}

We now bound from below the mean of $Y_n(y,z)$.
\begin{lemma}
\label{lem:meanlowerbound}
Assuming \eqref{eqn:supercritical}, \eqref{eqn:boundary} and \eqref{eqn:spine}, there exist $c>0$ and $Z \geq 1$ such that for all $y \in [0,n^{1/2}]$, $z \geq Z$, and $n \in \N$, $\E(Y_n(y,z)) \geq c(1+y)e^{-y}$.
\end{lemma}

\begin{proof}
Let $n \in \N$, $y \in [0,n^{1/2}]$ and $z \geq 1$. By Lemma~\ref{lem:firstmoment}, we have
\[  \E(Y_n(y,z))\geq \hat{\E}\left[ e^{-V(w_n)} \ind{w_n \in \mathcal{Y}_n(y,z)} \right]
  \geq n^{3/2} e^{-y} \hat{\P}(w_n \in \mathcal{Y}_n(y,z)).
\]
To bound this probability, we observe first that
\[
  \hat{\P}(w_n \in \bar{\calA}_n(y) \cap \calB_n(y,z)) = \hat{\P}(w_n \in \bar{\calA}_n(y)) - \hat{\P}(w_n \in \bar{\calA}_n(y) \cap \calB_n(y,z)^c),
\]
and $\hat{\P}(w_n \in \bar{\calA}_n(y)) \geq c (1+y) n^{-3/2}$ by Lemma~\ref{lem:excursionmarchealeatoirelb}. Moreover, by Lemma~\ref{lem:excursionSpine}, there exists $C>0$ such that
\[
  \hat{\P}(w_n \in \bar{\calA}_n(y) \cap \calB_n(y,z)^c) \leq C \frac{1+y}{n^{3/2}} \left( \P(\xi (w_0) \geq  z) + \hat{\E}((\xi(w_0) -z)_+)^2 \right).
\]
By \eqref{eqn:spine}, we have $\hat{\E}((\xi(w_0))_+^2)<+\infty$, therefore by dominated convergence, we have
\[
  \lim_{z \to +\infty} \sup_{n \in \N, y \geq 0} \frac{n^{3/2}}{1+y}\hat{\P}(w_n \in \bar{\calA}_n(y) \cap \calB_n(z)^c) = 0,
\]
thus for all $z \geq 1$ large enough we have $\hat{\P}(w_n \in \bar{\calA}_n(y) \cap \calB_n(z)) \geq c(1+y) n^{-3/2}$, which ends the proof.
\end{proof}

These two lemmas are used to complete the proof of Theorem~\ref{thm:tailestimate}.
\begin{proof}[Lower bound in Theorem~\ref{thm:tailestimate}]
By the Cauchy-Schwarz inequality, we have
\[ \P(Y_n(y,z) \geq 1) \geq \frac{\E(Y_n(y,z))^2}{\E(Y_n(y,z)^2)}. \]
Using Lemmas~\ref{lem:meanupperbound} and~\ref{lem:meanlowerbound}, there exists $z \geq 1$ such that
\[  \P(Y_n(y,z) \geq 1) \geq \frac{\left(c(1 + y)e^{- y}\right)^2}{C(1 + y) e^{z-y}}\geq c(1+y)e^{-y}.\]
As a consequence, we conclude $\P(M_n \geq m_n +y) \geq \P(G_n(y,z) \neq \emptyset) \geq c (1 + y)e^{- y}$.
\end{proof}

Lemma \ref{lem:meanupperbound} can also be used to prove that with high probability, individuals in the branching random walk above $m_n$ are either close relatives or their most recent common ancestor is not far from the root of the process.
\begin{theorem}
\label{thm:genealogy}
Under the assumptions \eqref{eqn:supercritical}, \eqref{eqn:boundary} and \eqref{eqn:spine}, we have
\[
  \lim_{R \to +\infty} \limsup_{n \to +\infty} \P\left( \exists |u|,|v|=n : V(u),V(v) \geq m_n, |u\wedge v| \in [R,n-R] \right) = 0.
\]
\end{theorem}

\begin{proof}
Let $n \geq 1$. We introduce for any $y,z \geq 0$ the random variable
\[
  X_{n}(y,z) = \sum_{|u|=n} \ind{V(u) \geq m_n} \ind{u \in \mathcal{A}_n(y) \cap \mathcal{B}_n(y,z)^c}.
\]
By Lemma \ref{lem:firstmoment}, we have
\begin{align*}
  \E(X_{n}(y,z)) &= \hat{\E}\left( e^{-V(w_n)} \ind{V(w_n) \geq m_n} \ind{w_n \in \mathcal{A}_n(y) \cap \mathcal{B}_n(y,z)^c} \right)\\
  &\leq (n+1)^{3/2} \hat{\P}\left( V(w_n) \geq m_n, w_n \in \mathcal{A}_n(y) \cap \mathcal{B}_n(y,z)^c \right)\\
  &\leq C (1 + y)^3\left( \hat{\P}(\xi(w_0) \geq z) + \hat{\E}\left((\xi(w_0)-z)^2_+\right) \right)
\end{align*}
using Lemma \ref{lem:excursionSpine}. We observe that for any $y, z, R \geq 0$:
\begin{align*}
  &\P\left( \exists |u|,|v|=n : V(u),V(v) \geq m_n, |u\wedge v| \in [R,n-R] \right)\\
  \leq &\P(\exists |u|<n : V(u) \geq f_n(|u|)+y) + \P( X_n(y,z) \neq 0 ) + \P\left( \exists k \in [R,n-R] : Y^{(2)}_{n,k}(y,z) \neq 0 \right)\\
  \leq &C (1 + y)e^{-y} + \E(X_n(y,z)) + \sum_{k=R}^{n-R} \E\left( Y^{(2)}_{n,k}(y,z) \right),
\end{align*}
using Lemma \ref{lem:upperfrontier}. We apply Lemma \ref{lem:meanupperbound}, we obtain
\begin{multline*}
  \limsup_{n \to +\infty} \P\left( \exists |u|,|v|=n : V(u),V(v) \geq m_n, |u\wedge v| \in [R,n-R] \right)\\
  \leq C \left( (1 + y)e^{-y} + (1+y)^3 \chi(z) + \frac{(1+y)e^{z-y}}{R^{1/2}} \right)
\end{multline*}
where $\chi(z) = \hat{\P}(\xi(w_0) \geq z) + \hat{\E}\left((\xi(w_0)-z)^2_+\right)$ satisfies $\lim_{z \to +\infty} \chi(z) = 0$ by \eqref{eqn:spine}. We set $z = \log \log R$ and $y = - \log \chi(z)$, letting $R \to +\infty$ concludes the proof.
\end{proof}

\section{Concentration estimates for the maximal displacement}
\label{sec:conclusion}

We prove in this section that Theorem~\ref{thm:tailestimate} implies Theorem~\ref{thm:main}. To do so, we use the fact that on the survival event $S$, the size of the population in the process grows at exponential rate, as in a Galton-Watson process. More precisely, we use the following result, which is a straightforward consequence of \cite{Mal15p}[Lemma 2.9].
\begin{lemma}
\label{lem:exponentialexplosion}
Let $(\T,V)$ be a branching random walk satisfying \eqref{eqn:supercritical} and \eqref{eqn:spine}. There exists $a>0$ and $\rho>1$ such that a.s. on $S$, for all $n \geq 1$ large enough, $\# \left\{ |u|=n : V(u) \geq - n a \right\} \geq \rho^{n}$.
\end{lemma}

\begin{proof}[Proof of Theorem~\ref{thm:main}]
We recall that $m_n = \frac{3}{2} \log n$, we prove that
\[
  \lim_{y \to +\infty} \limsup_{n \to +\infty} \P(|M_n - m_n| \geq y, S) = 0.
\]
Using the upper bound of Theorem~\ref{thm:tailestimate}, we have
\[
  \limsup_{n \to +\infty} \P(M_n \geq m_n + y) \leq C(1+y) e^{- y} \underset{y \to +\infty}{\longrightarrow} 0.
\]
To complete the proof, we have to strengthen the lower bound of Theorem~\ref{thm:tailestimate}, which states
\[
  \exists c > 0, \forall n \in \N, \forall y \in [0,n^{1/2}], \P(M_n \geq m_n + y) \geq c(1 + y)e^{- y}.
\]

To do so, we observe that by Lemma~\ref{lem:exponentialexplosion}, there exists $a>0$ and $\rho>1$ such that a.s. for any $k$ large enough, there are at least $\rho^k$ individuals above $-ka$. On this event, each individual starts an independent branching random walk, and the largest of their maximal displacement at time $n-k$ is smaller than $M_n$. Therefore for any $y \geq 0$ and $k \geq 1$, we have
\[
  \P(M_n \leq m_n - y)
  \leq \P\left( \# \{ |u|=k : V(u) \geq -ka \} < \rho^k \right) + \left( 1 - \P(M_{n-k} \geq m_n - y + ka) \right)^{\rho^k},
\]
thus $\limsup_{n \to +\infty} \P(M_n \leq m_n - y) \leq \P\left( \# \{ |u|=k : V(u) \geq -ka \} < \rho^k \right) + (1 - c (ka - y)_+e^{-ka})^{\rho^k}$.
We conclude the proof choosing $y = ka - 1$ and letting $k \to +\infty$.
\end{proof}

\appendix

\section{Proof of the random walk estimates}

\subsection{Proof of Lemmas~\ref{lem:excursionmarchealeatoireub} and~\ref{lem:excursionmarchealeatoirelb}}

The proofs of these two lemmas are based on a simple observation: an excursion of a random walk above a curve can be divided into three parts. The first third of the curve is a random walk staying above the line. The last third is the same process, but reversed in time. Finally, the middle curve has to connect the two extremities. Therefore, we expect
\[
  \P(T_n \in [z-y,z+h-y], T_j \geq -y, j \leq n) \approx \P(T_j \geq -y, j \leq n/3) \P(T_{n/3} \in [0,h]) \P(T_j \geq -z, j \leq n/3)
\]
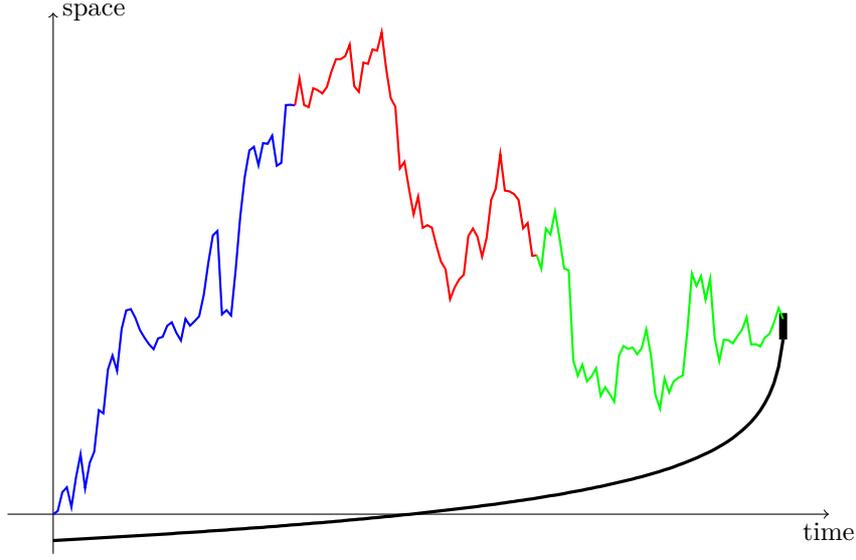
\begin{figure}[ht]
\centering
\begin{tikzpicture}[yscale=0.35, xscale=0.6]
\draw [line width=3 pt] (16.0,6.62) -- (16.0,7.62);
\draw [->] (-1,0) -- (17,0) node[below]{time};
\draw [->] (0,-1.5) -- (0,19) node[right]{space};
\draw [thick, color=blue] (0,0.0) -- (0.1,0.12) -- (0.2,0.83) -- (0.3,1.02) -- (0.4,0.27) -- (0.5,1.39) -- (0.6,2.27) -- (0.7,0.98) -- (0.8,1.94) -- (0.9,2.37) -- (1.0,3.94) -- (1.1,3.82) -- (1.2,5.48) -- (1.3,6.01) -- (1.4,5.43) -- (1.5,7.01) -- (1.6,7.72) -- (1.7,7.77) -- (1.8,7.44) -- (1.9,6.98) -- (2.0,6.69) -- (2.1,6.44) -- (2.2,6.25) -- (2.3,6.66) -- (2.4,6.72) -- (2.5,7.14) -- (2.6,7.27) -- (2.7,6.87) -- (2.8,6.58) -- (2.9,7.4) -- (3.0,7.14) -- (3.1,7.31) -- (3.2,7.5) -- (3.3,8.3) -- (3.4,9.52) -- (3.5,10.56) -- (3.6,10.73) -- (3.7,7.57) -- (3.8,7.73) -- (3.9,7.53) -- (4.0,9.26) -- (4.1,11.27) -- (4.2,12.79) -- (4.3,13.78) -- (4.4,13.92) -- (4.5,13.22) -- (4.6,14.06) -- (4.7,14.03) -- (4.8,14.34) -- (4.9,13.2) -- (5.0,13.31) -- (5.1,15.5) -- (5.2,15.51) -- (5.3,15.49);
\draw [thick, color=red] (5.3,15.49) -- (5.4,16.51) -- (5.5,15.5) -- (5.6,15.42) -- (5.7,16.14) -- (5.8,16.06) -- (5.9,15.94) -- (6.0,16.19) -- (6.1,16.78) -- (6.2,17.24) -- (6.3,17.24) -- (6.4,17.35) -- (6.5,17.79) -- (6.6,16.21) -- (6.7,16.0) -- (6.8,17.11) -- (6.9,17.06) -- (7.0,17.61) -- (7.1,17.56) -- (7.2,18.26) -- (7.3,16.86) -- (7.4,15.76) -- (7.5,15.44) -- (7.6,13.1) -- (7.7,13.35) -- (7.8,12.32) -- (7.9,11.36) -- (8.0,12.03) -- (8.1,10.85) -- (8.2,10.95) -- (8.3,10.86) -- (8.4,10.18) -- (8.5,9.57) -- (8.6,9.28) -- (8.7,8.15) -- (8.8,8.59) -- (8.9,8.9) -- (9.0,9.07) -- (9.1,10.54) -- (9.2,10.82) -- (9.3,10.51) -- (9.4,9.76) -- (9.5,10.46) -- (9.6,11.91) -- (9.7,12.33) -- (9.8,13.64) -- (9.9,12.26) -- (10.0,12.23) -- (10.1,12.13) -- (10.2,11.91) -- (10.3,10.83) -- (10.4,11.03) -- (10.5,9.79) -- (10.6,9.81);
\draw [thick, color=green] (10.6,9.81) -- (10.7,9.3) -- (10.8,10.82) -- (10.9,10.6) -- (11.0,11.46) -- (11.1,10.45) -- (11.2,9.31) -- (11.3,9.24) -- (11.4,5.83) -- (11.5,5.25) -- (11.6,5.67) -- (11.7,5.03) -- (11.8,5.22) -- (11.9,5.53) -- (12.0,4.48) -- (12.1,4.81) -- (12.2,4.55) -- (12.3,4.26) -- (12.4,5.99) -- (12.5,6.37) -- (12.6,6.26) -- (12.7,6.32) -- (12.8,6.06) -- (12.9,6.29) -- (13.0,7.0) -- (13.1,6.04) -- (13.2,4.52) -- (13.3,4.0) -- (13.4,5.14) -- (13.5,4.61) -- (13.6,5.03) -- (13.7,5.17) -- (13.8,5.25) -- (13.9,6.94) -- (14.0,9.12) -- (14.1,8.65) -- (14.2,9.02) -- (14.3,8.13) -- (14.4,8.92) -- (14.5,6.66) -- (14.6,5.8) -- (14.7,6.61) -- (14.8,6.59) -- (14.9,6.47) -- (15.0,6.75) -- (15.1,6.99) -- (15.2,7.46) -- (15.3,6.43) -- (15.4,6.44) -- (15.5,6.35) -- (15.6,6.69) -- (15.7,6.83) -- (15.8,7.26) -- (15.9,7.81) -- (16.0,7.4) ;
\draw [very thick] (0,-1.0) -- (0.1,-0.99) -- (0.2,-0.98) -- (0.3,-0.97) -- (0.4,-0.96) -- (0.5,-0.95) -- (0.6,-0.94) -- (0.7,-0.93) -- (0.8,-0.92) -- (0.9,-0.91) -- (1.0,-0.9) -- (1.1,-0.89) -- (1.2,-0.88) -- (1.3,-0.87) -- (1.4,-0.86) -- (1.5,-0.85) -- (1.6,-0.84) -- (1.7,-0.83) -- (1.8,-0.82) -- (1.9,-0.81) -- (2.0,-0.8) -- (2.1,-0.79) -- (2.2,-0.78) -- (2.3,-0.77) -- (2.4,-0.76) -- (2.5,-0.75) -- (2.6,-0.74) -- (2.7,-0.72) -- (2.8,-0.71) -- (2.9,-0.7) -- (3.0,-0.69) -- (3.1,-0.68) -- (3.2,-0.67) -- (3.3,-0.66) -- (3.4,-0.64) -- (3.5,-0.63) -- (3.6,-0.62) -- (3.7,-0.61) -- (3.8,-0.6) -- (3.9,-0.58) -- (4.0,-0.57) -- (4.1,-0.56) -- (4.2,-0.55) -- (4.3,-0.53) -- (4.4,-0.52) -- (4.5,-0.51) -- (4.6,-0.5) -- (4.7,-0.48) -- (4.8,-0.47) -- (4.9,-0.46) -- (5.0,-0.44) -- (5.1,-0.43) -- (5.2,-0.41) -- (5.3,-0.4) -- (5.4,-0.39) -- (5.5,-0.37) -- (5.6,-0.36) -- (5.7,-0.34) -- (5.8,-0.33) -- (5.9,-0.32) -- (6.0,-0.3) -- (6.1,-0.29) -- (6.2,-0.27) -- (6.3,-0.26) -- (6.4,-0.24) -- (6.5,-0.22) -- (6.6,-0.21) -- (6.7,-0.19) -- (6.8,-0.18) -- (6.9,-0.16) -- (7.0,-0.14) -- (7.1,-0.13) -- (7.2,-0.11) -- (7.3,-0.09) -- (7.4,-0.08) -- (7.5,-0.06) -- (7.6,-0.04) -- (7.7,-0.02) -- (7.8,-0.01) -- (7.9,0.01) -- (8.0,0.03) -- (8.1,0.05) -- (8.2,0.07) -- (8.3,0.09) -- (8.4,0.11) -- (8.5,0.13) -- (8.6,0.15) -- (8.7,0.17) -- (8.8,0.19) -- (8.9,0.21) -- (9.0,0.23) -- (9.1,0.25) -- (9.2,0.27) -- (9.3,0.29) -- (9.4,0.32) -- (9.5,0.34) -- (9.6,0.36) -- (9.7,0.38) -- (9.8,0.41) -- (9.9,0.43) -- (10.0,0.46) -- (10.1,0.48) -- (10.2,0.51) -- (10.3,0.53) -- (10.4,0.56) -- (10.5,0.58) -- (10.6,0.61) -- (10.7,0.64) -- (10.8,0.67) -- (10.9,0.7) -- (11.0,0.72) -- (11.1,0.75) -- (11.2,0.78) -- (11.3,0.82) -- (11.4,0.85) -- (11.5,0.88) -- (11.6,0.91) -- (11.7,0.95) -- (11.8,0.98) -- (11.9,1.02) -- (12.0,1.05) -- (12.1,1.09) -- (12.2,1.13) -- (12.3,1.17) -- (12.4,1.21) -- (12.5,1.25) -- (12.6,1.29) -- (12.7,1.33) -- (12.8,1.38) -- (12.9,1.42) -- (13.0,1.47) -- (13.1,1.52) -- (13.2,1.57) -- (13.3,1.62) -- (13.4,1.68) -- (13.5,1.73) -- (13.6,1.79) -- (13.7,1.86) -- (13.8,1.92) -- (13.9,1.99) -- (14.0,2.06) -- (14.1,2.13) -- (14.2,2.21) -- (14.3,2.29) -- (14.4,2.37) -- (14.5,2.46) -- (14.6,2.56) -- (14.7,2.66) -- (14.8,2.77) -- (14.9,2.89) -- (15.0,3.03) -- (15.1,3.17) -- (15.2,3.33) -- (15.3,3.5) -- (15.4,3.7) -- (15.5,3.93) -- (15.6,4.21) -- (15.7,4.54) -- (15.8,4.97) -- (15.9,5.58) -- (16.0,6.62) ;
\end{tikzpicture}
\caption{A random walk excursion above $k \mapsto f_n(k)$}
\end{figure}

\begin{proof}[Proof of Lemma~\ref{lem:excursionmarchealeatoireub}]
We denote by $p = \floor{n/3}$. Applying the Markov property at time $p$, we have
\begin{multline*}
  \P_y\left( T_n - f_n(n) \in [z-h,z], T_j \geq f_n(j), j \leq n \right) \\
  \leq \P_y\left( T_j \geq - A j^{\alpha}, j \leq p \right) \sup_{x \geq -f_n(p)} \P_x\left( T_n - f_n(n) \in [z-h,z], T_j \geq f_n(p+j), j \leq n-p \right).
\end{multline*}
We set $\hat{T}_k = T_{n-p} - T_{n-p-k}$, which is a random walk with the same law as $T$. Note that for any $x \in \R$;
\begin{align}
  &\P_x \left( T_{n-p} - f_n(n) \in [z-h,z], T_j \leq f_n(p+j), j \leq n-p \right)\nonumber\\
  &\qquad \leq \P\left( \hat{T}_{n-p}-f_n(n) \in [x+z-h,x+z], \hat{T}_j \leq z + f_n(n)-f_n(n-j), j \leq n-p \right)\nonumber\\
  &\qquad \leq \P_{-z} \left( T_{n-p}-f_n(n) \in [x-h,x], T_j \leq A j^\alpha, j \leq n-p \right) \label{eqn:inter}.
\end{align}
We apply again the Markov property at time $p$, we have
\begin{multline*}
  \qquad\P_{-z} \left( T_{n-p}-f_n(n) \in [x-h,x], T_j \leq A j^\alpha, j \leq n-p \right)\\
  \leq \P_{-z} \left( T_j \leq A j^\alpha, j \leq p\right) \sup_{x \in \R} \P\left( T_{n-2p} \in [x,x+h] \right).\qquad
\end{multline*}
Consequently, using \eqref{eqn:locallimitub} and \eqref{eqn:ballotub}, we obtain
\[
  \P_y\left( T_n - f_n(n) \in [z,z+h], T_j \geq f_n(j), j \leq n \right) \leq \frac{(1 + y \wedge n^{1/2})(1+ h \wedge n^{1/2})(1+z\wedge n^{1/2})}{n^{3/2}}.
\]
\end{proof}

\begin{proof}[Proof of Lemma~\ref{lem:excursionmarchealeatoirelb}]
The proof is very similar to the previous one. Let $p = \floor{n/3}$, we apply the Markov property at time $p$, to obtain, for any $\epsilon>0$ and $n \geq 1$ large enough,
\begin{multline*}
  \P_y\left( T_n \leq f_n(n) + H, T_j \geq f_n(j), j \leq n \right) \\
  \geq \P_y\left( T_p \in [p^{1/2},2p^{1/2}], T_j \geq \epsilon, j \leq p \right) \\
  \times\inf_{x \in [p^{1/2},2p^{1/2}]} \P_x\left( T_n \leq f_n(n)+H, T_j \geq f_n(p+j), j \leq n-p \right).
\end{multline*}
Setting again $\hat{T}_k = T_{n-p} - T_{n-p-k}$, for any $x \in \R$ we have
\begin{align*}
  &\P_x \left( T_{n-p} \geq f_n(n)+H, T_j \leq f_n(p+j), j \leq n-p \right)\\
  &\qquad \geq \P\left( \hat{T}_{n-p}-f_n(n) \in [x,x+H], \hat{T}_j \leq f_n(n)-f_n(n-j), j \leq n-p \right)\\
  &\qquad \geq \P \left( T_{n-p}-f_n(n) \in [x-h,x], T_j \leq f_n(n)-f_n(n-j), j \leq n-p \right).
\end{align*}
We apply again the Markov property at time $p$, for $x \in [p^{1/2},2p^{1/2}]$ we have
\begin{multline*}
  \qquad \P \left( T_{n-p}-f_n(n) \in [x-h,x], T_j \leq A j^\alpha, j \leq n-p \right)\\
  \geq \P\left( T_p \in [p^{1/2}, 2p^{1/2}], T_j \leq f_n(n)-f_n(n-j), j \leq p \right)\\
  \times \inf_{z \in [p^{1/2},2p^{1/2}]} \P\left( T_j \geq 0, T_p - f_n(n) \in [x-H,x] \right).\qquad
\end{multline*}
Thus, we apply \eqref{eqn:locallimitlb} and \eqref{eqn:ballotlb} to conclude the proof.
\end{proof}

\subsection{Proof of Lemma~\ref{lem:excursionSpine}}

We now extend Lemma \ref{lem:excursionmarchealeatoireub} to random walks enriched by additional random variables. The idea behind the proof is that a random walk excursion is typically at distance $O(n^{1/2})$ from the boundary. Therefore, as soon as $\E(\xi_+^2)$ is finite, we expect similar upper bounds for the branching random walk. To prove Lemma~\ref{lem:excursionSpine}, we first prove the following result, that is a direct consequence of \cite[Lemma C.1]{Aid13}
\begin{lemma}
\label{lem:ballotSpine}
With the notation of Lemma~\ref{lem:excursionmarchealeatoireub}, there exists $C>0$ such that for any $y \geq 0$,
\[
  \P\left(\exists k \leq n : T_k \leq f_n(k)-y + \xi_{k}, T_j \geq f_n(j)-y, j \leq n \right)\leq C\frac{1+y}{n^{1/2}}\left(\P(\xi_{1} \geq 0) + \E\left((\xi_1)_+^2\right) \right)
\]
\end{lemma}

\begin{proof}
Let $n \in \N$ and $y \geq 0$. For $k < n$ we denote by
\begin{align*}
  \pi_k &= \P_y\left(T_k \leq f_n(k) + \xi_{k}, T_j \geq f_n(j), j \leq n\right)\\
  &= \E_y\left( \ind{T_k \leq f_n(k) + \xi_{k}, T_j \geq f_n(j), j \leq k+1} \P_{T_{k+1}}\left( T_j \geq f_n(k+1+j), j \leq n-k-1 \right) \right)\\
  &\leq C \E_y\left( \ind{T_k \leq f_n(k) + \xi_{k}, T_j \geq f_n(j), j \leq k} \frac{1+(T_{k}-f_n(k))}{(n-k+1)^{1/2}} \right),
\end{align*}
by Markov property and \eqref{eqn:ballotub}. Then, conditioning on $(X_{k},\xi_{k})$ and applying Lemma \ref{lem:excursionmarchealeatoireub}, we have
\[
  \pi_k \leq \frac{C}{(n-k+1)^{1/2}} \E\left( \ind{\xi_{k} \geq 0}(1 + (\xi_{k})_+) \left( \frac{(1 + y)(1 + (\xi_{k}-X_k)_+^2 \wedge k)}{(k+1)^{3/2}} \right) \right).
\]

We denote by $(X,\xi)$ a random vector with the same law as $(X_k,\xi_k)$. We have
\begin{multline*}
  \P\left(\exists k \leq n : T_k \leq f_n(k)-y + \xi_{k}, T_j \geq f_n(j)-y, j \leq n \right) \leq \sum_{k=0}^{n-1} \pi_k\\
  \leq C (1+y) \E\left( \ind{\xi \geq 0}(1 + (\xi)_+) \sum_{k=0}^{n-1} \frac{1 + (\xi-X)_+^2 \wedge k}{k^{3/2}(n-k+1)^{1/2}} \right).
\end{multline*}
Observing that $\sum_{k=0}^{n/2} \frac{1 + (\xi-X)_+^2 \wedge k}{k^{3/2}} \leq C (1 + (\xi-X)_+)$, we conclude that
\[
  \P\left(\exists k \leq n : T_k \leq f_n(k)-y + \xi_{k}, T_j \geq f_n(j)-y, j \leq n \right) \leq C\frac{1+y}{n^{1/2}}\left(\P(\xi_{1} \geq 0) + \E\left((\xi_1)_+^2\right) \right).
\]
\end{proof}

\begin{proof}[Proof of Lemma~\ref{lem:excursionSpine}]
We use Lemma \ref{lem:ballotSpine} to prove this result. Recalling that
\[
  \tau = \inf\{ k < n : T_k \leq f_n(k)+ \xi_{k+1} \} = \inf\{ k < n : T_{k+1} \leq f_n(k) + (\xi_{k+1}+X_{k+1}) \},
\]
we have
\begin{multline*}
  \qquad \P_y\left(T_n \leq f_n(n)+H, \tau < n, T_j \geq f_n(j), j \leq n\right)\\
   \leq \P_y\left(T_n \leq f_n(n)+H, \tau \leq n/2, T_j \geq f_n(j), j \leq n\right) \\
   + \P_y\left(T_n \leq f_n(n)+H, \tau \in (n/2,n], T_j \geq f_n(j), j \leq n\right).\qquad
\end{multline*}

To bound the first term, we apply the Markov property at time $n/2$, we have
\begin{align*}
  &\P_y\left(T_n \leq f_n(n)+H, \tau \leq n/2, T_j \geq f_n(j), j \leq n\right)\\
  \leq &\P_y \left( \tau < n/2, T_j \geq f_n(j), j \leq n/2 \right) \sup_{z \in \R} \P_z \left( T_n \leq f_n(n) + H, T_j \geq f_n(n/2+j), j \leq n/2 \right)\\
  \leq & C\frac{1+y}{n^{1/2}}\left(\P(\xi_{1}+X_1 - 1 \geq 0) + \E\left((\xi_1+X-1)_+^2\right) \right) \times \frac{1}{n},
\end{align*}
by Lemma \ref{lem:ballotSpine} and \eqref{eqn:inter}.

We denote by $\hat{T}_j = T_n - T_{n-j}$ and $\hat{\xi}_j = \xi_{n-j}$. To bound the second term, we observe that
\begin{align*}
  &\P_y\left(T_n \leq f_n(n)+H, \tau \in (n/2,n], T_j \geq f_n(j), j \leq n\right)\\
  \leq & \P\left( \begin{array}{l}\hat{T}_n - f_n(n)+y \in [0,H], \hat{T}_j \leq f_n(n)-f_n(n-j), j \leq n,\\
   \exists k < n/2 : \hat{T}_k \geq f_n(n)-f_n(n-k)-\hat{\xi}_{k-1}\end{array}\right)\\
   \leq & \P\left( T_n - f_n(n) + y \in [0,H], T_j \leq f_n(n)- f_n(n-j), \exists k < n/2 : T_k \geq f_n(n)-f_n(n-k) - \xi_{k-1} \right)\\
  \leq & C \frac{1+y}{n^{3/2}} \left(\P(\xi_{1}+X_1-1 \geq 0) + \E\left((\xi_1 + X_1-1)_+^2\right) \right),
\end{align*}
using the previous inequalities, which concludes the proof.
\end{proof}

\bibliographystyle{plain}

\end{document}